\documentclass[11pt,a4,leqno]{amsart}
\usepackage{amsmath,epsfig,graphicx,color,enumerate}
\usepackage{comment}
\usepackage{ascmac}
\usepackage{mathrsfs}
\usepackage{tikz}
\usepackage{subcaption}
\definecolor{darkblue}{rgb}{.2, 0.2,.8}
\definecolor{carageen}{rgb}{0,0.5,0.3}
\definecolor{darkred}{rgb}{1, 0,0}

\newcommand{\clt}{central limit theorem}

\newtheorem{lemma}{Lemma}[section]

\newtheorem{theorem}[lemma]{Theorem}

\renewcommand{\P}{{\mathbb P}}
\newtheorem{proposition}[lemma]{Proposition}
\newtheorem{definition}[lemma]{Definition}
\newtheorem{corollary}[lemma]{Corollary}
\newtheorem{example}[lemma]{Example}
\newtheorem{exercise}[lemma]{Exercise}
\newtheorem{remark}[lemma]{Remark}

\newtheorem{tab}[lemma]{Table}

\newcommand{\bth}{\begin{theorem}}
\newcommand{\ethe}{\end{theorem}}

\newcommand{\bre}{\begin{remark}\em }
\newcommand{\ere}{\end{remark}}

\newcommand{\ble}{\begin{lemma}}
\newcommand{\ele}{\end{lemma}}

\newcommand{\bde}{\begin{definition}}
\newcommand{\ede}{\end{definition}}
\newcommand{\bco}{\begin{corollary}}
\newcommand{\eco}{\end{corollary}}

\newcommand{\bpr}{\begin{proposition}}
\newcommand{\epr}{\end{proposition}}

\newcommand{\bexer}{\begin{exercise}}
\newcommand{\eexer}{\end{exercise}}

\newcommand{\bexam}{\begin{example}\rm }
\newcommand{\eexam}{\end{example}}

\newcommand{\efi}{\end{fig}}

\newcommand{\btab}{\begin{tab}}
\newcommand{\etab}{\end{tab}}

\newcommand{\var}{{\rm var}}

\newcommand{\corr}{{\rm corr}}

\newcommand{\beao}{\begin{eqnarray*}}
\newcommand{\eeao}{\end{eqnarray*}\noindent}

\newcommand{\beam}{\begin{eqnarray}}
\newcommand{\eeam}{\end{eqnarray}\noindent}

\newcommand{\beqq}{\begin{equation}}
\newcommand{\eeqq}{\end{equation}\noindent}

\newcommand{\bce}{\begin{center}}
\newcommand{\ece}{\end{center}}

\newcommand{\barr}{\begin{array}}
\newcommand{\earr}{\end{array}}

\newcommand{\vague}{\stackrel{\lower0.2ex\hbox{$\scriptscriptstyle
                    \it{v} $}}{\rightarrow}}
\newcommand{\weak}{\stackrel{\lower0.2ex\hbox{$\scriptscriptstyle
                    \it{w} $}}{\rightarrow}}
\newcommand{\what}{\stackrel{\lower0.2ex\hbox{$\scriptscriptstyle
                    \it{\hat{w}} $}}{\rightarrow}}

\newcommand{\bdis}{\begin{displaymath}}
\newcommand{\edis}{\end{displaymath}\noindent}

\newcommand{\N}{\mathbb{N}}
\newcommand{\R}{\mathbb{R}}

\newcommand{\ov}{\overline}
\newcommand{\ud}{\underline}
\newcommand{\wt}{\widetilde}
\newcommand{\wh}{\widehat}
\newcommand{\vep}{\varepsilon}

\newcommand{\Z}{{\mathbb Z}}

\newcommand{\seq}{sequence}

\newcommand{\E }{{\mathbb E}}
\renewcommand{\P }{{\mathbb P}}

\allowdisplaybreaks
\makeatletter
\@namedef{subjclassname@2020}{\textup{2020} Mathematics Subject Classification}
\makeatother

\begin{document}
\bibliographystyle{plain}
\title[Improvement of central limit theorem for stationary $\rho$-mixing sequences]
{Can we further improve the central limit theorem for stationary $\rho$-mixing sequences $?$}
\author[M. Matsui]{Muneya Matsui}
\thanks{Muneya Matsui's research is partly supported by the JSPS Grant-in-Aid for Scientific Research C
(26K14736).} 
\address{Graduate School of Economics, The university of Osaka, 1-7, Machikaneyama, Toyonaka, Osaka 560-0043, Japan.}
\email{mmuneya@gmail.com}
\today

\begin{abstract}
Through the seminal works of Peligrad and Bradley, the existing sufficient conditions for the central limit theorem for 
$\rho$-mixing sequences are known to be very close to necessary. Nevertheless, we show that further improvements to these conditions 
are achievable when the $\rho$-mixing coefficients are non-summable, bridging a long-standing gap in the borderline case. 
The key idea is an iterative block-size technique that enables precise evaluations of the slowly varying variance components. 
Finally, known examples and new extensions--including iterated logarithmic mixing rates--are investigated 
to demonstrate the sharpness of our criteria, revealing a clear parameter structure and posing related open problems. 
\end{abstract}
\maketitle

\section{Introduction} 
Throughout this paper, we use the following notation. 
Let $\log x$ denote $\log_2(x)$ (base $2$) and $\log^+x = \log (x\vee 1)$. 
For two sequences $(a_n)$ and $(b_n)$, we write $a_n \sim b_n$ as $n\to \infty$ if $\lim_{n\to\infty} a_n/b_n=1$. 
The notation $a_n \ll b_n$ means $a_n = O(b_n)$. The integer part of a real number $y>0$ is denoted by $[y]$.
Let $X:=(X_t)_{t\in \Z}$ be a strictly stationary sequence of random variables defined on a probability space $(\Omega,\mathcal{F},\P)$. 
Denote the partial sum of $X$ by $S_n=\sum_{i=1}^n X_i$ and its variance by $\sigma_n^2 = \var (S_n)$. 
For $-\infty \le m\le n\le \infty$, let $\mathcal{F}^n_m$ denote the $\sigma$-field generated by $(X_t)_{m\le t\le n}$. 
The dependence coefficient is defined by  
\[
 \rho(n) = \sup_{f\in L_2(\mathcal{F}^0_{-\infty}), \, g\in L_2(\mathcal{F}^{\infty}_n)} |\corr(f,g)|.
\]
A stationary sequence $X$ is said to be ``$\rho$-mixing'' if $\rho(n)\to 0$ as $n\to\infty$.
Furthermore, $N(0,1)$ denote the standard normal distribution.  

We start with the following landmark central limit theorem (CLT) for $\rho$-mixing sequences established by Ibragimov \cite{ibragimov:1975}.

\begin{theorem}[Theorems 1 and 2 in \cite{ibragimov:1975}]
\label{thm:Ibragimov}
 Suppose that $(X_t)$ is a strictly stationary sequence of random variables such that 
\begin{align}
\begin{split}\label{condition:Ibragimov}
\E X_0 =0,&\quad \E X_0^2 <\infty,\quad \sigma_n\to\infty\quad \text{as}\quad n\to\infty \\ 
\text{and}&\quad \rho(n)\to 0\quad \text{as}\quad n\to\infty.
\end{split}
\end{align}
Then $\sigma_n^2=nh(n)$, where $h(n)$ is a slowly varying function as $n\to\infty$. Suppose, in addition, that at least one of the following two conditions is satisfied:
\begin{align*}
(\mathrm{i})  & \hspace{2cm} \E |X_0|^{2+\delta} <\infty\quad \text{for some}\quad \delta>0,\\ 
(\mathrm{ii}) & \hspace{2cm} \sum_{i=1}^\infty \rho (2^i)<\infty.  
\end{align*}
Then $S_n/\sigma_n\stackrel{d}{\to} N(0,1)$ as $n\to\infty$. 
\end{theorem}
Bradley \cite{bradley:1980} showed that conditions $(\mathrm{i})$ and $(\mathrm{ii})$ cannot be omitted altogether. 
He later conjectured in \cite{bradley:1987} that a suitable combination of modified versions of $(\mathrm{i})$ and 
$(\mathrm{ii})$ would improve the CLT of Theorem \ref{thm:Ibragimov}. Peligrad \cite{peligrad:1987} elegantly resolved 
this conjecture, obtaining the following theorem. Let $g:[0,\infty) \to [0,\infty)$ be a non-decreasing function. 

\begin{theorem}[Theorem 1 in Peligrad \cite{peligrad:1987}]
\label{peligrad:1987:clt:thm}
 Suppose that $(X_t)$ is a strictly stationary sequence satisfying \eqref{condition:Ibragimov} and 
\begin{align}
& \E X_0^2 g(|X_0|) <\infty, \label{condi:i:modified:1}\\
& g(n^{1/2}) \gg \exp\Big( 2\sum_{i=1}^{[\log n]}\rho (2^i)/(1-\eta) \Big)\quad \text{for some}\quad 0<\eta<1. \label{condi:ii:modified:1}
\end{align}
Then $S_n/\sigma_n\stackrel{d}{\to} N(0,1)$ as $n\to\infty$. 
\end{theorem}

It is easy to see that $\exp\big(c\sum_{i=1}^{[\log n]}\rho(2^i)\big)$ is a slowly varying function as $n\to\infty$ 
for any $c\in \R$, meaning that $g$ in Theorem \ref{peligrad:1987:clt:thm} can be chosen to be a slowly varying function. 
Note that Theorem \ref{peligrad:1987:clt:thm} implies Theorem \ref{thm:Ibragimov}: when $g$ is a constant function, 
condition $(\mathrm{ii})$ is satisfied, whereas when $g(x)=x^\delta$ for some $\delta>0$, condition $(\mathrm{i})$ follows. 
Furthermore, condition \eqref{condi:ii:modified:1} accommodates the case where $\sum_{i=1}^\infty \rho(2^i)=\infty$ 
at the expense of an extra slowly varying term $g(|X_0|)$ multiplied by $X^2_0$ in \eqref{condi:i:modified:1}. 
In this case, if we take $g$ to be a slowly varying function, neither $(\mathrm{i})$ nor $(\mathrm{ii})$ 
of Theorem \ref{thm:Ibragimov} holds. From Theorem \ref{peligrad:1987:clt:thm}, one can deduce the following results:

\begin{corollary}[Corollaries 1 and 2 in \cite{peligrad:1987}]
\label{peligrad:1987:clt:cor12}
 Let $(X_t)$ be a strictly stationary sequence satisfying \eqref{condition:Ibragimov}. Consider the following conditions: \\
$(a)$ Assume that for some $0<\vep<1$ and $\alpha>0$, 
\begin{align*}
(a1) \qquad & \E X_0^2 (\log^+ |X_0|)^{2\alpha/(1-\vep)} <\infty, \\
(a2) \qquad & \rho(n) \le \alpha (\log n)^{-1} \quad \text{for every $n$ sufficiently large.}
\end{align*}
$(b)$ Assume that for some $0<\beta <1$, $0<\vep<1$, and $\alpha >0$, 
\begin{align*}
 (b1) \qquad & \E X_0^2 \exp\Big( \frac{2\alpha}{1-\beta}\frac{1}{1-\vep} (2\log^+|X_0|)^{1-\beta} \Big) <\infty, \\ 
 (b2) \qquad & \rho(n) \le \alpha (\log n)^{-\beta} \quad \text{for every $n$ sufficiently large.} 
\end{align*}
If either $(a)$ or $(b)$ is satisfied, then $S_n/\sigma_n\stackrel{d}{\to} N(0,1)$ as $n\to\infty$. 
\end{corollary}

As noted in Remarks 1 and 2 of \cite{peligrad:1987} (see also Corollaries 1 and 2 in \cite{bradley:1987}), 
the conditions in Corollary \ref{peligrad:1987:clt:cor12} are remarkably sharp and almost necessary. 
Thus, there appeared to be little room for further refinement. Nevertheless, we assert that a genuine 
improvement is indeed achievable. Specifically, we show that the parameter $\vep$ in condition $(a1)$, 
as well as that in condition $(b1)$ provided $\beta > 1/2$, can be completely removed (see Corollary \ref{cor:clt:log:mixing}). 
Furthermore, our framework extends seamlessly to even slower mixing regimes, such as iterated logarithmic decay rates 
(see Corollary \ref{cor:clt:new:log:mixing}), where it reveals a remarkably sharp correspondence between 
the required moment conditions and the decay rates of the mixing coefficients.

The key idea lies in a refined asymptotic analysis of the slowly varying component 
$h(n)$ in the variance representation $\sigma_n^2 = n h(n)$ from Theorem \ref{thm:Ibragimov}. 
In Peligrad \cite{peligrad:1987}, the extra factor $\vep>0$ in Corollary \ref{peligrad:1987:clt:cor12} 
inherently arises from the crude geometric index truncation $2^{(1-\vep_1)i}$ used in the upper and lower 
bounds for $h(n)$ (see Lemma \ref{lem:1:peligrad:1987}). To eliminate this artifact, we replace $2^{-\vep_1 i}$ 
with a delicate sequence of auxiliary slowly varying functions. 
Crucially, these auxiliary functions incorporate the slowly varying bounds of $h(u)$ originally derived in \cite{peligrad:1987}, 
thereby establishing a self-improving iterative scheme for $h(n)$. 
By integrating these novel components into our evaluation mechanism, 
we significantly sharpen the analytical framework of \cite{peligrad:1987}, enabling us to handle delicate borderline behavior that was previously inaccessible.

We conclude this section with a brief overview of the related literature on $\rho$-mixing sequences. 
Invariance principles in this framework were studied by Peligrad \cite{peligrad:1982} and Shao \cite{shao:1988,shao:1989}. 
Meanwhile, the CLT under infinite variance was established by Bradley \cite{bradley:1988} and subsequently extended to 
an invariance principle by Shao \cite{shao:1993}.

The remainder of this paper is organized as follows. In Section \ref{sec:main:results}, 
we state our main results and compare them with the preceding literature. The proofs are provided 
in Section \ref{sec:proofs}, while auxiliary results are deferred to Appendix \ref{appendix}.

\section{Main results}  
\label{sec:main:results}
For the reader's convenience we begin by explaining the core idea penetrating this paper. 
We take up Lemma 1 of \cite{peligrad:1987} and explain the idea.
\begin{lemma}
\label{lem:1:peligrad:1987}
 Suppose that $(X_t)_{t\in \Z}$ satisfies \eqref{condition:Ibragimov}. Let $0<\vep_1 < \vep_2<1$. Then there exist two positive 
constants $\ov c_1= \ov c_1 (\rho,\vep_1)$ and $\ov c_2=\ov c_2(X,(\vep_i))$ such that for every $n\ge 1$
\begin{align}
 \sigma_n^2 & \le \ov c_1 n \E X_0^2 \exp\bigg(
\sum_{i=1}^{[(1-\vep_1) \log n]} \rho(2^i)/(1-\vep_1)
\bigg) =: n \ov h (n), \\
 \sigma_n^2 & \ge \ov c_2 n \exp \bigg(
-\sum_{i=1}^{[(1-\vep_1) \log n]} \rho(2^i)/(1-\vep_2)
\bigg) =: n\ud h(n). 
\end{align}
\end{lemma} 
Notice that the exponential parts of $\ov h,\ud h$ control the slowly varying part of $\sigma_n^2$ (cf. $h(x)$ in Theorem \ref{thm:Ibragimov}) and 
$\ov h,\ud h$ themselves are slowly varying. Indeed, the constant $\eta$ in Theorem \ref{peligrad:1987:clt:thm}, 
so that $\vep$ in Corollary \ref{peligrad:1987:clt:cor12} partly originates from $(\vep_i)$\footnote{
The constant $\eta$ depends also on $\vep$ in Lemma 3 of \cite{peligrad:1987} where 
the upper bound for the $4$th moment $\E|S_n|^4 \le C\big(n^{1+\vep}+\sigma_n^4\big)$ is deduced. 
By more accurate analysis we replace $n^{\vep}$ with slowly varying functions of 
$n$. The method is similar to that done for improvement of $\ov h,\ud h$ in Lemma \ref{lem:1:peligrad:1987}.  
}.
Therefore, in order to remove $\vep$ in Corollary \ref{peligrad:1987:clt:cor12}, we need more accurate analysis 
of the parts related with $(\vep_i)$. More precisely we focus on $2^{-\vep_1 i}$ terms in 
the following equivalent expressions to $\ov h,\ud h$: 
\begin{align}
\label{exp:ovh:udh:respective}
 \ov h(n)  = \wh c_1 \E X_0^2 \exp\bigg(
\sum_{i=1}^{[\log n]} \rho([2^{(1-\vep_1)i}]) 
\bigg), \quad 
 \ud h(n)  = \wh c_2 \exp \bigg(
-\sum_{i=1}^{[\log n]} \rho([2^{(1-\vep_1)i}])/(1-\beta)
\bigg), 
\end{align}
where $(1-\vep_1)(1-\beta)=(1-\vep_2)$, and replace $2^{-\vep_1 i}$ with more accurate 
slowly varying functions of $i$. 

Therefore, to state our main results, we need several slowly varying functions of $n$, which
 inevitably become complex for our purpose. 
Let $q_0$ 
be an arbitrary non-increasing slowly varying function such that 
$\sum_{i=1}^\infty q_0^{1/2}(2^i)< \infty$ holds. 
Our main theorem relies on the following two non-decreasing slowly varying functions,  
\begin{align*}
 h_\delta (n) = \exp\Big(5 \sum_{i=1}^{[\log n]} \rho^{2/(2+\delta)}(2^{i})\Big),\quad \delta\in(0,1),\quad n\ge 1
\end{align*}
and 
\begin{align*}
 h_d (n) = \exp\Big(- 2\sum_{i=k_0}^{[\log n]} \log \big\{ 1-
\rho\big(2_\ast^i \big) \big\}
\Big), \quad n\ge 2^{k_0} 
\end{align*}
where 
$k_0$ is an integer, 
and $\ud \ell$ is  
a non-increasing slowly varying function given 
by $\ud \ell(n)= \ud h^4(n) q_0(n)$, and $2_\ast^i=[2^i \ud \ell(2^i)]$ (see Lemma \ref{lem:lower:bound}).  

We remark that $h_d$ is a sharper lower bound than $\ud h$.
\begin{theorem}
\label{thm:main1}
Let $\delta$ be a positive constant in $(0,1)$. 
 Suppose that $(X_t)_{t\in \Z}$ is a strictly stationary sequence satisfying \eqref{condition:Ibragimov} and 
\begin{align*}
 \E X_0^2 g(|X_0|)<\infty
\end{align*}
and 
\begin{align}
\label{condi:ii:improved}
 g\Big(n^{1/2} \big(h_\delta(n)\, h_d^{(2+\delta)/4}(n) \big)^{-1/\delta}\Big) \gg h_d(n). 
\end{align}
Then $S_n/\sigma_n\stackrel{d}{\to} N(0,1)$ as $n\to\infty$.
\end{theorem}

Notice that Theorem \ref{thm:main1} includes Theorem \ref{peligrad:1987:clt:thm}.
Indeed, since $g$ is non-decreasing and $h_d$ is an asymptotically sharper lower bound than $\ud h$, we have 
\[
g(n^{1/2}) \ge g\Big(n^{1/2} \big(h_\delta(n)\, h_d^{(2+\delta)/4}(n) \big)^{-1/\delta}\Big)\quad 
\text{and}\quad 
\ud h(n) \ll h_d(n), 
\]
and thus \eqref{condi:ii:improved} implies \eqref{condi:ii:modified:1}.
Moreover, we have the following simpler Corollary than Corollary \ref{peligrad:1987:clt:cor12}. 

\begin{corollary}
\label{cor:clt:log:mixing}
Let $(X_t)$ be a strictly stationary sequence satisfying \eqref{condition:Ibragimov} and let $\alpha>0$ and $\beta\in (1/2,1)$.
Assume either $(a)$ or $(b)$. 
\begin{align*}
(a). \qquad & (a1)\qquad \E X_0^2 (\log^+ |X_0|)^{2\alpha} <\infty, \\
           & (a2)\qquad \rho(n) \le \alpha(\log n)^{-1}\quad \text{for every $n$ sufficiently large}.
\end{align*}
\begin{align*}
(b). \qquad & (b1) \qquad  \E X_0^2 \exp\Big[
2\alpha \big(2\log^+ |X_0|\big)^{1-\beta}/(1-\beta)
\Big] <\infty, \\
& (b2) \qquad  \rho(n) \le \alpha(\log n)^{-\beta}\quad \text{for every $n$ sufficiently large}.
\end{align*}
Then $S_n/\sigma_n\stackrel{d}{\to} N(0,1)$ as $n\to\infty$. 
\end{corollary}

We give the proof of Corollary \ref{cor:clt:log:mixing} in the next section only for the case $(b)$ since that 
for the case $(a)$ is similar and easier. In the following remarks we denote $\log^{\vee 2}x = \log (x\vee 2)$ for further convenience. 

\begin{remark}
\label{rem:log(n)^(-1)}
\noindent $(\mathrm{i})$ According to Bradley \cite[Corollary 1]{bradley:1987}, there exists a strictly stationary sequence $X=(X_t)$ satisfying \eqref{condition:Ibragimov}, $\rho(n)\ll (\log n)^{-1}$, and $\mathbb{E}[ X_0^2 (\log^+ |X_0|)^\gamma ] <\infty$ for some $\gamma>0$, for which the CLT fails. Peligrad \cite{peligrad:1987} pointed out that Corollary \ref{peligrad:1987:clt:cor12} implies $\sup_{n} \rho(n)\log n > \gamma/2$ must hold in this setting, but the borderline case $\sup_{n} \rho(n)\log n = \gamma/2$ remained open. Corollary \ref{cor:clt:log:mixing} resolves this issue by showing that the CLT indeed holds at this borderline, as illustrated in Figure \ref{fig:parameter_regions} $($A$)$.

\vspace{2mm}
\noindent $(\mathrm{ii})$ In the setting of $(\mathrm{i})$, we specifically assume that $X$ satisfies $\rho(n) = \alpha (\log^{\vee 2} n)^{-1}$ for some $\alpha>0$. To apply Bradley \cite[Theorem 1]{bradley:1987}, we set $q(x) = x^2 (\log^+ x)^\gamma$ and $\tau(n) = \alpha (\log^{\vee 2} n)^{-1}$ in condition $(1.5)$ therein. Using the non-decreasing property of $q$, we evaluate condition $(1.5)$ for sufficiently large $n$ as follows:
\begin{align*}
q\Big( n^{1/2} \exp\Big( -\frac{d}{2} \sum_{k=1}^n k^{-1} \tau(k) \Big) \Big) 
&\le q\big(n^{1/2} (\log n)^{-d\alpha/2}\big) \\
&= n (\log n)^{-d\alpha} \Big( \log \big( n^{1/2} (\log n)^{-d\alpha/2} \big) \Big)^\gamma \\
&= 2^{-\gamma} n (\log n)^{\gamma - d\alpha} + o\big(n (\log n)^{\gamma - d\alpha}\big) \quad \text{as } n \to \infty.
\end{align*}
Therefore, for the last expression to be $o(n)$, the inequality $\gamma/\alpha < d$ must hold, where $d$ is a given constant in \cite[Theorem 1]{bradley:1987}. 

Since Corollary \ref{cor:clt:log:mixing} ensures that the CLT holds in the region $\gamma \ge 2\alpha$, a natural 
subsequent question is: in which subspace of $\gamma < 2\alpha$ can one construct a counterexample to the CLT 
$($see Figure \ref{fig:parameter_regions} $($A$))$. We merely note that, in light of Theorem 1 and Proposition 0 of 
\cite{bradley:1987}, the constant $d$ appears to depend on $A_j$ and $B_\tau$ defined therein, and a further analysis of these constants is required.
\end{remark}

\begin{remark}
\label{rem:log(n)^(-beta)}
\noindent $(\mathrm{i})$ It seems difficult to extend Corollary \ref{cor:clt:log:mixing} $(b)$ to the region $\beta \in (0, 1/2)$.

\vspace{2mm}
\noindent $(\mathrm{ii})$ Bradley \cite[Corollary 2]{bradley:1980} proved that for any $\beta, \gamma > 0$ with $\beta + \gamma \le 1$, there exists a strictly stationary sequence $X=(X_t)$ satisfying \eqref{condition:Ibragimov}, $\rho(n) \ll (\log n)^{-\beta}$, and $\E[ X_0^2 (\log^+ |X_0|)^\gamma ] < \infty$ for which the CLT fails. In contrast, Peligrad \cite{peligrad:1987} pointed out that Corollary \ref{peligrad:1987:clt:cor12} implies that 
the CLT holds for any $\gamma > 0$ and $\beta \in (0,1)$ satisfying $\beta + \gamma > 1$.

Our Corollary \ref{cor:clt:log:mixing} provides further insight into the borderline case $\beta + \gamma = 1$ with $\beta \in (1/2, 1)$. First, we replace conditions $(b1)$ and $(b2)$ respectively with equivalent conditions 
\begin{align*}
(b1)\qquad & \E\Big[ |X_0|^2 \exp\Big( (\log^+ |X_0|)^\gamma \Big) \Big] < \infty, \\
(b2)\qquad & \rho(n) \le \gamma 2^{-(\gamma+1)} (\log^{\vee 2} n)^{-\beta}\quad \text{for every $n$ sufficiently large}.
\end{align*}
If we specifically set $\rho(n) = \alpha (\log^{\vee 2} n)^{-\beta}$, the CLT holds provided that the coefficient $\alpha$ satisfies $\alpha \le \gamma 2^{-(\gamma+1)}$ for $\gamma \in (0, 1/2)$. 

To illustrate the sharpness of this borderline, we verify that a non-Gaussian limit $($or a failure of the CLT$)$ can indeed occur even when $\beta + \gamma = 1$. Following the notation in Bradley \cite[Theorem 1]{bradley:1980}, we set $\tau(n) = \alpha (\log^{\vee 2} n)^{-\beta}$ and $q(x) = x^2 \exp((\log^+ x)^\gamma)$. Applying the relation $\beta + \gamma = 1$ $($i.e., $1 - \beta = \gamma$$)$, we evaluate the following expression for sufficiently large $n$:
\begin{align*}
q\Big( n^{1/2} &\exp\Big( -\frac{d}{2} \sum_{k=1}^n k^{-1} \tau(k) \Big) \Big) \\
&\le q\Big( n^{1/2} \exp\Big( -\frac{\alpha d}{2(1-\beta)} (\log n)^{1-\beta} \Big) \Big) \\
&= n \exp\Big( -\frac{\alpha d}{\gamma} (\log n)^\gamma + 2^{-\gamma} (\log n)^\gamma \Big\{ 1 - \frac{\alpha d}{\gamma} (\log n)^{\gamma-1} \Big\}^\gamma \Big).
\end{align*}
For a fixed $\gamma$, we can choose $\alpha$ large enough such that $\alpha d > \gamma 2^{-\gamma}$, 
which ensures that the last expression becomes $o(n)$ as $n \to \infty$. 
This also holds with $\alpha d \ge \gamma 2^{-\gamma}$ if $\gamma > 1/2$. 
Consequently, condition $(1.5)$ of Bradley \cite[Theorem 1]{bradley:1980} 
is satisfied. Since other conditions on $\tau$ and $q$ follow for $\gamma, \beta \in (0,1)$, 
we confirm that a counterexample to the CLT exists precisely at this borderline. 

In summary, when $\beta+\gamma=1$, the parameter space in the $(\alpha, \gamma)$-plane can be separated 
into three distinct parts: the CLT holds when $\alpha \le \gamma 2^{-(\gamma+1)}$ for $\gamma \in (0,1/2)$; 
Bradley's counterexample condition remains applicable when $\alpha > \gamma 2^{-\gamma} / d$ for $\gamma, \beta \in (0,1)$ and 
$\alpha = \gamma 2^{-\gamma} / d$ for $\gamma > 1/2$; 
and the intermediate space constitutes an unresolved gap that depends on the specific value of the constant $d>0$ 
$($see Figure \ref{fig:parameter_regions} $($B$))$. 
\end{remark}

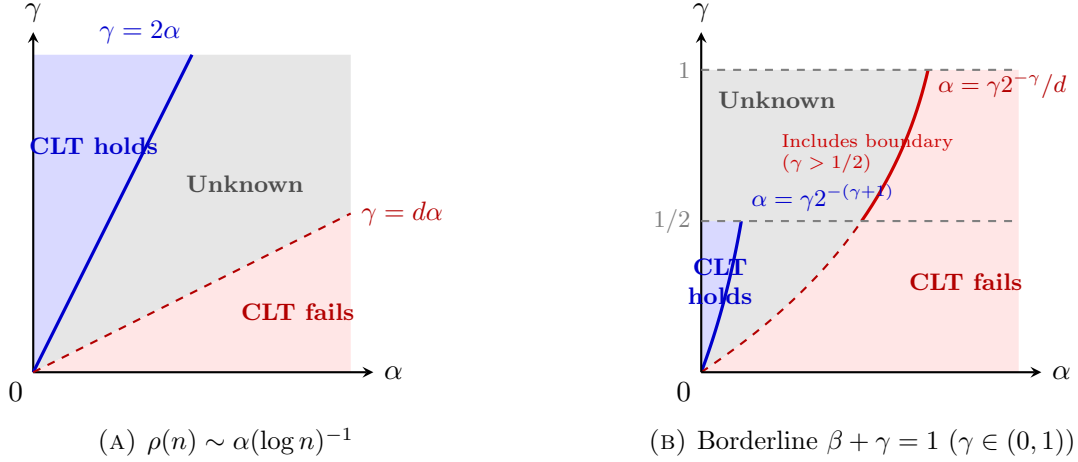
\begin{figure}[htbp]
\centering
\begin{subfigure}[b]{0.49\textwidth}
\centering
\begin{tikzpicture}[scale=1.0]
    \fill[blue!15] (0,0) -- (2.1,4.2) -- (0,4.2) -- cycle;

    \fill[gray!20] (0,0) -- (2.1,4.2) -- (4.2,4.2) -- (4.2,2.1) -- cycle;

    \fill[red!10] (0,0) -- (4.2,2.1) -- (4.2,0) -- cycle;

    \draw[->, >=stealth, thick] (0,0) -- (4.5,0) node[right] {$\alpha$};
    \draw[->, >=stealth, thick] (0,0) -- (0,4.5) node[above] {$\gamma$};
    \node[below left] at (0,0) {$0$};

    \draw[very thick, blue!80!black] (0,0) -- (2.1,4.2) node[above left, font=\small] {$\gamma = 2\alpha$};

    \draw[thick, dashed, red!70!black] (0,0) -- (4.2,2.1) node[right, font=\small] {$\gamma = d\alpha$};

    \node[align=center, blue!80!black, font=\footnotesize] at (0.8, 3.0) {\textbf{CLT holds}};
    \node[align=center, gray!60!black, font=\footnotesize] at (2.8, 2.5) {\textbf{Unknown}};
    \node[align=center, red!70!black, font=\footnotesize] at (3.5, 0.8) {\textbf{CLT fails}};
\end{tikzpicture}
\caption{$\rho(n) \sim \alpha (\log n)^{-1}$}
\end{subfigure}
\hfill
\begin{subfigure}[b]{0.49\textwidth}
\centering
\begin{tikzpicture}[scale=1.0]
    \begin{scope}
        \clip (0,0) rectangle (4.2, 4.2);

        \fill[blue!15] (0,0) 
            -- plot[domain=0:0.5, samples=30] ({3.0*\x*pow(2, -\x-1)}, {4.0*\x}) 
            -- (0,2.0) -- cycle;

        \fill[gray!20] (0,2.0) 
            -- (0,4.0) 
            -- plot[domain=1.0:0, samples=40] ({6.0*\x*pow(2, -\x)}, {4.0*\x}) 
            -- plot[domain=0:0.5, samples=30] ({3.0*\x*pow(2, -\x-1)}, {4.0*\x}) 
            -- (0,2.0) -- cycle;

        \fill[red!10] plot[domain=0:1.0, samples=40] ({6.0*\x*pow(2, -\x)}, {4.0*\x}) 
            -- (4.2,4.0) -- (4.2,0) -- cycle;

        \draw[very thick, blue!80!black, domain=0:0.5, samples=30] 
            plot ({3.0*\x*pow(2, -\x-1)}, {4.0*\x});

        \draw[thick, dashed, red!70!black, domain=0:0.5, samples=20] 
            plot ({6.0*\x*pow(2, -\x)}, {4.0*\x});

        \draw[very thick, red!80!black, domain=0.5:1.0, samples=20] 
            plot ({6.0*\x*pow(2, -\x)}, {4.0*\x});
    \end{scope}

    \draw[->, >=stealth, thick] (0,0) -- (4.5,0) node[right] {$\alpha$};
    \draw[->, >=stealth, thick] (0,0) -- (0,4.5) node[above] {$\gamma$};
    \node[below left] at (0,0) {$0$};

    \draw[dashed, thick, black!50] (0, 2.0) node[left, font=\footnotesize] {$1/2$} -- (4.2, 2.0);
    \draw[dashed, thick, black!50] (0, 4.0) node[left, font=\footnotesize] {$1$} -- (4.2, 4.0);

    \node[blue!80!black, font=\footnotesize, above right] at (0.53, 2.0) {$\alpha = \gamma 2^{-(\gamma+1)}$};
    \node[red!70!black, font=\footnotesize, right] at (3.0, 3.8) {$\alpha = \gamma 2^{-\gamma}/d$};
    
    \node[red!80!black, font=\tiny, align=left] at (2.2, 2.9) {Includes boundary\\($\gamma > 1/2$)};

    \node[align=center, blue!80!black, font=\footnotesize] at (0.25, 1.2) {\textbf{CLT}\\ \textbf{holds}};
    \node[align=center, gray!60!black, font=\footnotesize] at (1.0, 3.6) {\textbf{Unknown}};
    \node[align=center, red!70!black, font=\footnotesize] at (3.5, 1.2) {\textbf{CLT fails}};
\end{tikzpicture}
\caption{Borderline $\beta+\gamma=1$ ($\gamma \in (0, 1)$)}
\end{subfigure}

\caption{Parameter regions for the central limit theorem in the $(\alpha, \gamma)$-plane. The blue regions 
represent the parameter space where the CLT is established by Corollary \ref{cor:clt:log:mixing}, while 
the red regions indicate where the conditions for Bradley's counterexamples are satisfied. The intermediate 
gray regions represent the unresolved gaps, which depend on the unspecified constant $d>0$. 
Panel (A) corresponds to the case $\rho(n) \sim \alpha (\log n)^{-1}$ discussed in Remark \ref{rem:log(n)^(-1)}. 
Panel (B) illustrates the borderline case $\beta + \gamma = 1$ with $\gamma \in (0, 1)$ and 
$\rho(n) \sim \alpha (\log^{\vee 2} n)^{-\beta}$ detailed in Remark \ref{rem:log(n)^(-beta)}, where the boundary 
$\alpha = \gamma 2^{-\gamma}/d$ is included in the counterexample region for $\gamma \in (1/2, 1)$.}
\label{fig:parameter_regions}
\end{figure}

In somewhat different examples, our results again yield clear relations between 
moments and the decay rates of the mixing coefficients. These examples further 
clarify the essence of the open problems in the extreme borderline cases.

\begin{corollary}
\label{cor:clt:new:log:mixing}
Let $(X_t)$ be a strictly stationary sequence satisfying \eqref{condition:Ibragimov} and let $\alpha>0$ and $\beta <1$.
Assume either $(a)$ or $(b)$. 
\begin{align*}
(a). \qquad & (a1)\qquad \E \big[X_0^2 (\log^+ \log^+ |X_0|)^{2\alpha}\big] <\infty, \\
           & (a2)\qquad \rho(n) \le \alpha(\log n)^{-1}(\log \log n)^{-1}\quad \text{for every $n$ sufficiently large}.
\end{align*}
\begin{align*}
(b). \qquad & (b1) \qquad  \E \Big[ X_0^2 \exp\big(
2\alpha \big(\log^+ \log^+ |X_0|\big)^{1-\beta}/(1-\beta)
\big)\Big] <\infty, \\
& (b2) \qquad  \rho(n) \le \alpha (\log n)^{-1}(\log \log n)^{-\beta} \quad \text{for every $n$ sufficiently large}.
\end{align*}
Then $S_n/\sigma_n\stackrel{d}{\to} N(0,1)$ as $n\to\infty$. 
\end{corollary}

Since the proof of Corollary \ref{cor:clt:new:log:mixing} follows almost identically 
to that of Corollary \ref{cor:clt:log:mixing}, we omit the details here (they are available upon request). 

\begin{remark}
\noindent $(\mathrm{i})$ According to Bradley \cite[Theorem 1]{bradley:1980}, 
if we set alternative conditions to those in Remark \ref{rem:log(n)^(-1)}$(\mathrm{i})$ such that 
$\rho(n) \ll (\log n)^{-1}(\log \log n)^{-1}$ and $\E [X_0^2 (\log^{\vee 2}\log^+ |X_0|)^\gamma] <\infty$ 
for some $\gamma\in (0,1)$, then there exists a strictly stationary sequence $X$ satisfying these conditions 
for which the CLT fails.\footnote{Notice that the condition $q(xy)\le q(x)q(y)$ in \cite[Theorem 1]{bradley:1980} 
is explicitly satisfied for $\gamma\in (0,1)$. While it is unclear if this holds for $\gamma>1$, this may not be a 
strict limitation; see \cite[p.99]{bradley:1980} for details.}
In view of Corollary \ref{cor:clt:new:log:mixing}$(a)$, this implies that 
$\sup_n \rho(n) (\log n) (\log\log n) > \gamma/2$ must hold in this counterexample setting.

\vspace{2mm}
\noindent $(\mathrm{ii})$ In addition to $(\mathrm{i})$, assume specifically that $X$ satisfies 
$\rho(n)=\alpha (\log^{\vee 2} n)^{-1}\big(\log^{\vee 2} \log^+ n \big)^{-1}$. By setting $q(x)=x^2 (\log^{\vee 2}\log^{+} x )^\gamma$ 
and $\tau(n)= \alpha (\log^{\vee 2} n)^{-1}\big(\log^{\vee 2} \log^+ n\big)^{-1}$, the same consideration as in 
Remark \ref{rem:log(n)^(-1)}$(\mathrm{ii})$ applies. The borderline for the validity and failure of the CLT takes 
an analogous structure to Figure \ref{fig:parameter_regions}$($A$)$, where the precise value of the constant $d>0$ again becomes crucial. 
\end{remark}

\begin{remark}
\noindent $(\mathrm{i})$ By applying \cite[Theorem 1]{bradley:1980}, we can verify that for any $\beta<1$ and $\gamma\in(0,1)$ with 
$\beta+\gamma \le 1$, there exists a strictly stationary sequence $X$ satisfying \eqref{condition:Ibragimov}, 
$\rho (n) \ll (\log n)^{-1} (\log \log n)^{-\beta}$, and $\E [X_0^2 \exp \big( ( \log^+\log^+ |X_0|)^\gamma \big)] <\infty$ 
for which the CLT fails. On the other hand, in view of Corollary \ref{cor:clt:new:log:mixing}$(b)$, the CLT holds for 
any $\beta<1$ and $\gamma>0$ satisfying $\beta+\gamma>1$. Therefore, the borderline case $\beta+\gamma=1$ $($with $\beta<1$$)$ 
emerges again as a critical problem.

To investigate this borderline, we first replace conditions $(b1)$ and $(b2)$ with the equivalent borderline specifications:
\begin{align*}
& (b1') \qquad  \E X_0^2 \exp\Big[\big(\log^+ \log^+ |X_0|\big)^{\gamma}
\Big] <\infty, \\
& (b2') \qquad  \rho(n) \le \gamma/2\, (\log n)^{-1}(\log \log n)^{-\beta} \quad \text{for every $n$ sufficiently large}.
\end{align*}
If we specifically set $\rho(n)=\alpha (\log^{\vee 2}n)^{-1} \big(\log^{\vee 2}\log^{+} n \big)^{-\beta}$ and 
$q(x)= x^2 \exp\big((\log^+\log^+ x)^\gamma \big)$, we can apply Bradley's construction. A counterexample to 
the CLT can be established in the region $\{\alpha > \gamma / d\} \cap \{\gamma\in (0,1)\}$ for a given 
constant $d>0$ $($subject to the same footnote condition regarding $q(xy)\le q(x)q(y)$$)$.

In summary, when $\beta+\gamma=1$, the parameter space in the $(\alpha,\gamma)$-plane separates into three distinct regions: 
the CLT holds when $\alpha \le \gamma/2$; Bradley's counterexample condition is applicable when $\alpha > \gamma/d$ $($for $\gamma<1$$)$; 
and the intermediate space remains unresolved. This structural trichotomy is fundamentally identical to the situation illustrated 
in Figure \ref{fig:parameter_regions}$($A$)$. 

\vspace{2mm}
\noindent$(\mathrm{ii})$
We briefly note that the analytical argument used here $($see the proof of Corollary \ref{cor:clt:log:mixing}$)$ 
can be systematically adapted to handle slower decay rates involving higher-order iterations of logarithms. 
By repeating essentially the same procedure, one can theoretically extend the results to an arbitrary 
depth of iterated logarithms $L_k(n)$ $($where $L_1(n) = \log n, L_2(n) = \log \log n, \dots$$)$ 
without any additional restrictions other than $\beta<1$.
\end{remark}

\section{Proofs}
\label{sec:proofs}
In proofs $c$ (without any accents) denotes an arbitrary positive constant whose value is not of interest.
When we need to distinguish positive constants within a proof, we write $c_1,c_2,\ldots$. 
Denote $S_n(m)= X_{m+1}+\cdots+ X_{n+m}$. We also denote $\sigma_\cdot=\sigma(\cdot)$ for simplicity of notation in the following.
We set $\prod_{i=t}^{s}=1$ if $t>s$ as regards the product notation. 

\begin{lemma}
\label{lemm:upper:bound}
 Suppose that $(X_t)$ satisfies \eqref{condition:Ibragimov}. For $\vep_1 \in (0,1)$ define 
\[
 \ov h(n)= \ov c_1 \E X_0^2 \exp\bigg(
\sum_{i=1}^{[(1-\vep_1) \log n]} \rho(2^i)/(1-\vep_1)
\bigg),
\]
where $\ov c_1( \rho, \vep_1)$ is a positive constant. 
Let $q(n)$ be a non-increasing slowly varying function such that 
$\sum_{i=1}^\infty q^{1/2}(2^i)<\infty$. 
Define a non-increasing slowly varying function by   
\[
 \ov \ell (m)= q(m) \ov h^{-1}(m).  
\]
Then, there exists a positive constant $\wt c_1 = \wt c_1\big( X,q,\vep_1 \big)$ such that for every $n\ge 1$
\[
 \sigma_n^2 \le \wt c_1 n \E X_0^2 \exp\bigg(\sum_{i=1}^{[\log n]} \rho \big(
[\,2^i \ov \ell (2^i)\,]
\big)
\bigg). 
\]
\end{lemma}

Notice that $X:=(X_t)$ in $\wt c_1$ determines the mixing sequence $\rho=(\rho_i)$ and we 
omit the dependence on $\rho$ in $\wt c_1$. 

\begin{proof}
 Recall from Theorem \ref{thm:Ibragimov} that there exists a slowly varying function $h(n)$ such that 
 $\sigma_n^2 =nh(n)$ holds for all $n\in\N$. 
 Therefore, for all integers $m,p\in\N$, we have 
\begin{align}
\label{ineq:sigma:rho:1}
 |\sigma_{2m}-\|S_m(0)+S_m(m+p)\|_2| \le 2\sigma_p = 2p^{1/2}h^{1/2}(p).
\end{align}
Moreover, by the definition of $\rho$-mixing coefficient 
\begin{align}
 \label{ineq:sigma:rho:2}
 |\E (S_m(0)+S_m(m+p))^2 -2 \sigma_m^2| \le 2\rho(p) \sigma_m^2. 
\end{align}
From the last two inequalities and the fact that $h\le \ov h$, 
we obtain 
\begin{align}
\label{ineq:upper:recurrence}
 \sigma_{2m} \le 2^{1/2} \big(1+\rho(p)\big)^{1/2} \sigma_m + 2p^{1/2} \ov h^{1/2}(p).
\end{align}
Since $p = [m \ov \ell(m)]$ and $\ov \ell(m) = q(m) \ov h^{-1}(m)$, the slow variation of $\ov h$ implies that
\begin{equation*}
    p^{1/2} \ov h^{1/2}(p) \le m^{1/2} q^{1/2}(m)
\end{equation*}
holds for all sufficiently large $m$. Combining this with \eqref{ineq:upper:recurrence}, we obtain \eqref{ineq:upper:recurrence:2}.
\begin{align}
 \sigma_{2m} 
& \le 2^{1/2} \big(
1+\rho([m \ov \ell (m)])
\big)^{1/2}\sigma_m + 2 m^{1/2} q^{1/2}(m). \label{ineq:upper:recurrence:2}
\end{align}
We further put $m=2^{j-1},\, j\in \N$ and write $2^i_\dag = [2^i \ov \ell (2^i)],\,i\in \{0\}\cup \N$. 
Then, by repeatedly applying \eqref{ineq:upper:recurrence:2}, we arrive at 
\begin{align*}
 \sigma(2^j) &\le \prod_{i=k}^{j-1}\big(
1+\rho(2^i_\dag)
\big)^{1/2} 2^{(j-k)/2} \sigma(2^k) \\
&\quad + 2 \sum_{i=1}^{j-k} 2^{(i-1)/2} \prod_{l=j-i+1}^{j-1} \big(1+\rho(2^l_\dag)\big)^{1/2} 2^{(j-i)/2} q^{1/2}(2^{j-i}) \\
& \le 2^{j/2} \prod_{i=k}^{j-1}\big(
1+\rho (2^i_\dag)
\big)^{1/2} \Big(\sigma(2^k)2^{-k/2}+2^{1/2} \sum_{i=1}^{j-1} q^{1/2}(2^i)\Big). 
\end{align*}
Due to the summability of $q^{1/2}$, we see by putting $k=1$ that 
there is a constant $c_1=c_1\big(X,q,\vep_1 \big)$ such that 
\[
 \sigma(2^j) \le c_1 \prod_{i=0}^{j-1}\big(
1+\rho (2^i_\dag ) 
\big)^{1/2} \cdot 2^{j/2} \sigma_1. 
\]
By the relation $1+x \le e^{x}$ for every $x>0$, we obtain
\[
 \sigma^2(2^j) \le c_1^2\, 2^j\, \E X_0^2\, \exp\bigg(
 \sum_{i=0}^{j-1} \rho \big(
[\,2^i \ov \ell (2^i)\,]
\big)
\bigg). 
\]
Now by writing $n$ in binary form and evaluating the error, we reach 
\[
 \sigma^2(n) \le \wt c_1 n\, \exp\bigg(
\sum_{i=1}^{[\log n]} \rho \big(
[\,2^i \ov \ell (2^i)\,]
\big)
\bigg) \E X_0^2. 
\]
\end{proof}

\begin{lemma}
\label{lem:lower:bound}
 Suppose that $(X_t)$ satisfies \eqref{condition:Ibragimov}. 
Let $q_0(n)$ be a non-increasing slowly varying function such that 
$\sum_{i=1}^\infty q^{1/2}_0 (2^i)<\infty$.
Define a non-increasing slowly varying function $\ud \ell$ by $\ud \ell = \ud h^4 q_0 $.  
Let $k_0\in \N$ be such that $\rho\big([2^k \ud \ell (2^k)]\big)<1$ for all $k\ge k_0$. Then there exists a positive constant 
$\wt c_2=\wt c_2\big(X,(\vep_i)_{i=1,2},k_0,q_0 \big)$ such that for all $n \ge 2^{k_0}$
\begin{align*}
 \sigma_n^2 \ge \wt c_2\,n\,\exp\bigg(
\sum_{i=k_0}^{[\log n]} \log \big\{ 1-\rho\big([2^i \ud \ell (2^i)]\big) \big\}
\bigg). 
\end{align*}
\end{lemma}

Notice that the choice of $q$ and $q_0$ is flexible and we may take  
a uniform function $q_0=q$ in Lemmas \ref{lemm:upper:bound}, \ref{lem:lower:bound}.

\begin{proof}
Let $m,p \in \N$ such that $m\ge p$. 
 From the two inequalities \eqref{ineq:sigma:rho:1} and \eqref{ineq:sigma:rho:2} in Lemma \ref{lemm:upper:bound}, 
\begin{align}
\label{eq:sigma:lower}
 \sigma_m \le 2^{-1/2} \big(1-\rho (p)\big)^{-1/2} (\sigma_{2m}+ 2\sigma_p).
\end{align}
Substituting $m=2^k$ and $p= [m \ud \ell (m)]=[2^k \ud \ell (2^k)]$ into \eqref{eq:sigma:lower}, 
and iterating the inequality, we arrive at  
\begin{align}
\begin{split}
\label{eq:sigma:lower:recussion}
 \sigma(2^k) &\le 2^{(k-r)/2} \sigma(2^r) \prod_{i=k}^{r-1} \big(
1-\rho \big( 2^i_\ast \big)
\big)^{-1/2} \\
&\quad + 2^{1/2}\sum_{i=0}^{r-k-1} 2^{-i/2} \prod_{j=k}^{k+i} \big(
1-\rho \big( 2^{j}_\ast \big)
\big)^{-1/2}\,\sigma\big( 2^{k+i}_\ast \big),
\end{split}
\end{align}
where we write $2^i_\ast = [ 2^i \ud \ell (2^i)]$ for convenience. 
Let us examine the second term in the right handside more closely. We use the fact that for every 
$\beta>0$ there exists $x_\beta$ such that $(1-x)> \exp\big(-x/(1-\beta)\big)$ holds for all 
$0<x<x_\beta$. 
Choose $k_1$ such that $\rho\big(2^{k}_\ast \big)<1 \wedge x_\beta$ and 
$\ud \ell(2^k) \ge 2^{-\vep_1 k}$ for all $k\ge k_1$. This choice is permissible 
since $\ud \ell$ is a non-increasing slowly varying function \cite[Proposition 1.3.6 (v)]{bingham:goldie:teugels:1987}. 
Then for $k\ge k_1$ 
\begin{align*}
 \prod_{j=k}^{k+i} \big(
1-\rho \big(2^{j}_\ast \big) \big)^{-1/2} 
&= \exp \Big(-2^{-1} \sum_{j=k}^{k+i} \log \big(1-\rho (2^{j}_\ast) \big) \Big) \\
&\le \exp \bigg( 2^{-1} \sum_{j=k}^{k+i} \rho (2^{j}_\ast )/(1-\beta) \bigg) \\
&\le \exp \bigg( 2^{-1} \sum_{j=k}^{k+i} \rho ([2^{j(1-\vep_1)}] )/(1-\beta) \bigg) \\
&\le \wh c_2^{1/2} \ud h^{-1/2}(2^{k+i}). 
\end{align*}
For the term $\sigma(2_\ast^{k+i})$, we apply the inequality $\sigma_n^2 = nh(n) \le n\ov h(n)$ and 
exploit the relation between $\ov h$ and $\ud h$ from expressions in \eqref{exp:ovh:udh:respective}.
We deduce 
\begin{align*}
 \sigma(2^{k+i}_\ast) &\le 2^{(k+i)/2}\, \ud \ell^{1/2}(2^{k+i})\, \ov h^{1/2}\big(2^{k+i}\ud \ell(2^{k+i}) \big) \\
& \le 2^{(k+i)/2}\, \ud \ell^{1/2}(2^{k+i})\,\ov h^{1/2}(2^{k+i}) \\
& =2^{(k+i)/2} \ud \ell^{1/2}(2^{k+i})\, \ud h^{-1/2}(2^{k+i})\,(\wh c_1 \wh c_2)^{1/2} \sigma_1 \exp\bigg(
-\frac{\beta}{2(1-\beta)} \sum_{j=1}^{k+i}{\rho([2^{(1-\vep_1)j}])}
\bigg) \\
& =:2^{(k+i)/2} \ud \ell^{1/2}(2^{k+i}) \,\ud h^{-1/2}(2^{k+i}) c_1(k),
\end{align*}
where we utilize the fact that $\ov h$ is non-decreasing and $\ud \ell$ is non-increasing.  
Since $\sum_{i=1}^{[\log n]} \rho([2^{(1-\vep_1)i}])\to \infty$ $(n\to\infty)$, we may take $c_1(k)$ arbitrarily small 
by choosing $k$ sufficiently large, noticing that $c_1(k)$ is non-increasing w.r.t. $k$. 
Substituting these bounds into the second term of \eqref{eq:sigma:lower:recussion}, we obtain  
\begin{align*}
& 2^{1/2} \sum_{i=0}^{r-k-1} 2^{-i/2} \prod_{j=k}^{k+i} \big(1-\rho (2^{j}_\ast) \big)^{-1/2}\, \sigma(2^{k+i}_\ast) \\
&\le (2\wh c_1)^{1/2} c_1(k) 2^{k/2} \sum_{i=0}^{r-k-1} \ud\ell^{1/2}(2^{k+i}) \ud h^{-1}(2^{k+i}) \\
&\le (2 \wh c_1)^{1/2} c_1(k) 2^{k/2}\, \ud h(2^k)\, \sum_{i=0}^{r-k-1} q_0^{1/2}(2^{k+i}) \\ 
&\le  c_0(k) \,2^{k/2 }\,\ud h(2^k),  
\end{align*}
where we may let $0<c_1=c_1(k) <1$ if $k\ge k_2$ for some sufficiently large $k_2$.
Thus the inequality \eqref{eq:sigma:lower:recussion} implies that 
\[
 \sigma(2^k) \le 2^{(k-r)/2} \sigma(2^r) \prod_{i=k}^{r-1}\big(
1-\rho(2^i_\ast)
\big)^{-1/2} + c_1\, 2^{k/2}\,\ud h(2^k)
\]
holds for $k\ge k_1 \vee k_2 =:k_3$. 
Recalling  
that $\sigma(2^k)=2^{k/2} h^{1/2}(2^k)$ for 
a slowly varying function $h$, 
we can absorb the second term to deduce the existence of a positive constant 
$ c_2 = c_2 (X,(\vep_i),q_0,k)$ such that for every $r>k \ge k_3$,
\[
 \sigma(2^k) \le c_2\, 2^{(k-r)/2} \sigma(2^r) \prod_{i=k}^{r-1} 
\big(1-\rho(2^i_\ast) \big)^{-1/2}. 
\]
Rearranging this inequality yields
\begin{align*}
 \sigma^2(2^r) &\ge  c_2^{-2} \,2^{r-k} \sigma^2(2^k) \prod_{i=k}^{r-1} \big(
1-\rho(2^i_\ast)\big) \\
& \ge  c_3\,2^r \exp\bigg(
\sum_{i=k_0}^{r-1} \log \big(
1-\rho(2^i_\ast )
\big)\bigg), \qquad r\ge k_3,
\end{align*}
where $ c_3 = c_3(X, (\vep_i),k_3,q_0)$. 
By interpolating between $n$ and its nearest binary power, we obtain
\begin{align}
\label{pf:ineq:lower:lem:23:pre}
 \sigma^2(n) \ge c_4 n \exp\bigg(
\sum_{i=k_0}^{[\log n]}
 \log \big(
1-\rho(2^i_\ast )
\big)
\bigg)
\end{align}
for every $n\ge 2^{k_3}$ with some $c_4=c_4(X, (\vep_i),k_3,q_0)>0$.   
Finally, in order for \eqref{pf:ineq:lower:lem:23:pre} to hold across the entire range $2^{k_0}\le n<2^{k_3}$, we put 
\[
 \wt c_2 = \min\Big\{
c_4,\,\min_{2^{k_0}\le j<2^{k_3}}\sigma^2(j)j^{-1}
\Big\}
\]
since the exponential term $\exp(\cdot)\le 1$ in \eqref{pf:ineq:lower:lem:23:pre}.
\end{proof}

Let $p\in (2,4]$ and let $q_1 (n)$ be a non-increasing slowly varying function such that $\sum_{i=1}^\infty q_1 (2^i)<\infty$. 
We have the following bound for $\E|S_n|^p$.

\begin{lemma}
\label{lem:inq:4th:moment}
 Suppose that $(X_t)$ satisfies \eqref{condition:Ibragimov} and $\E|X_0|^p<\infty$. 
Then there exists a positive constant $\wt c_3=\wt c_3 (\rho,p,q_1)$ such that 
\begin{align}
\label{inq:4th:moment}
 \E|S_n|^p \le \wt c_3 \bigg(n \exp\Big( p 2^{p-2}
\sum_{i=1}^r \rho^{2/p}\big( 2_\ddag^{i/p} \big)
\Big)\E|X_0|^p+ \sigma_n^p \bigg),
\end{align}
where $r=[\log n]$ and $2_\ddag^{i/p}=[2^{i/p} q_1(2^i)],\,i\ge 1$. 
\end{lemma}

\begin{proof}
Let $\ell\in \N$ and let $\|X\|_p$ denote the norm $\|X\|_p=(\E|X|^p)^{1/p}$. 
For convenience, we write $a_m= \big(\E |S_m|^p \big)^{1/p}=\|S_m\|_p$. 
Obviously, we have 
\begin{align}
\label{ineq:pth:norm}
 \|S_k(2m)\|_p \le \|S_k(m)+S_{k+\ell+m}(m)\|_p +2 \ell a_1. 
\end{align}
We apply the elementary inequality (Lemma \ref{aux:lem:ineq})
\[
 (x+y)^p \le x^p + y^p + p 2^{p-2} (x^{p-1}y + x y^{p-1}),\quad p\ge 2,\quad x,y\ge 0
\]
to $\|S_k(m)+S_{k+\ell+m}(m)\|_p^p$ and obtain  
\begin{align}
\label{ineq:pth:moment}
\begin{split}
& \E|S_k(m)+S_{k+\ell+m}(m)|^p \\
& \le 2a_m^p+ p2^{p-2} 
\E\big[
|S_k(m)|^{p-1}|S_{k+\ell+m}(m)|+|S_k(m)||S_{k+\ell+m}(m)|^{p-1}
\big].
\end{split}
\end{align}
By Minkowski's inequality and the definition of $\rho$-mixing, we deduce  
\begin{align*}
 \E|S_k(m)|^{p-1}|S_{k+\ell+m}(m)| 
& \le \|S_k(m)\|_p^{p-2}\|S_k(m)\cdot S_{k+\ell+m}(m)\|_{p/2} \\
& \le a_m^{p-2}\big(\rho(\ell)a_m^p +\sigma_m^p\big)^{2/p} \\
& \le \rho^{2/p}(\ell) a_m^p +\sigma_m^2 a_m^{p-2},  
\end{align*}
where in the second step we exploit the stationarity and apply
\begin{align*}
 \E|S_k(m)S_{k+\ell+m}(m)|^{p/2} 
&= \E\big[\big(|S_k(m)|^{p/2}-\E|S_k(m)|^{p/2}\big)\big(|S_{k+\ell+m}(m)|^{p/2}-\E|S_{k+\ell+m}(m)|^{p/2}\big)\big] \\
&\qquad + (\E|S_k(m)|^{p/2})^{2} \\
& \le \rho(\ell) \E|S_k(m)|^p + (\E|S_k(m)|^2)^{p/2}. 
\end{align*}
Similarly, we obtain  
\begin{align*}
 \E |S_k(m)||S_{k+\ell+m}(m)|^{p-1} \le \rho^{2/p}(\ell) a_m^p + \sigma_m^2 a_m^{p-2}. 
\end{align*}
Substituting these results into \eqref{ineq:pth:moment} yields 
\begin{align*}
 \E|S_k(m)+S_{k+\ell+m}(m)|^p & 
\le 2 a_m^p + p2^{p-1} \big( \rho^{2/p}(\ell) a_m^p +\sigma_m^2 a_m^{p-2} \big) \\
& \le \big\{
2^{1/p} \big(1+ p2^{p-2} \rho^{2/p}(\ell)\big)^{1/p} a_m + 2^{(p+1)/2} \sigma_m 
\big\}^p. 
\end{align*}
This, together with \eqref{ineq:pth:norm}, yields the inequality 
\[
 a_{2m} \le 2^{1/p} \big(1+p 2^{p-2} \rho^{2/p}(\ell)\big)^{1/p} a_m + 2^{(p+1)/2} \sigma_m + 2\ell a_1. 
\]
We put $m=2^{r-1}$ and $\ell =[2^{(r-1)/p} q_1(2^{r-1})]= 2^{(r-1)/p}_\ddag$. 
Then 
\[
 a(2^r) \le 2^{1/p} \big(1+p 2^{p-2} \rho^{2/p} (2_\ddag^{(r-1)/p})\big)^{1/p}
a(2^{r-1}) + 2^{(p+1)/2}\sigma(2^{r-1}) + 2\cdot 2_\ddag^{(r-1)/p} a_1. 
\]
By recursion we obtain 
\begin{align}
\label{ineq:pf:lemma:4th}
\begin{split}
 a(2^r) &\le 2^{(r-k)/p} \prod_{i=k}^{r-1} \big(1+p2^{p-2} \rho^{2/p}(2_\ddag^{i/p})\big)^{1/p} a(2^k) \\
& \quad + \sum_{i=k}^{r-1} 2^{(r-i-1)/p} \prod_{l=i+1}^{r-1} \big(1+p2^{p-2} \rho^{2/p}(2_\ddag^{l/p})\big)^{1/p}
\Big(2^{(p+1)/2}\sigma(2^{i})+2\cdot 2_\ddag^{i/p} a_1 \Big).
\end{split} 
\end{align}
On the right-hand side, we exploit the inequality (\cite[Lemma 2]{peligrad:1987}): 
\[
\sigma(2^i) \le 2^{(r-i)(\vep -1)/2} \sigma(2^r) + c_1 \sigma_1,\quad  i<r,\quad 0<\vep<1, 
\]
where $c_1=c_1(\rho,\vep)$, and we specifically choose $0<\vep< (p-2)/(2p)$, yielding 
\begin{align*}
& \sum_{i=k}^{r-1} 2^{(r-i-1)/p} \prod_{l=i+1}^{r-1}
\big(1+p2^{p-2} \rho^{2/p}(2_\ddag^{l/p})\big)^{1/p} 2^{(p+1)/2} \sigma(2^i) \\
& \le 2^{(p+1)/2} \sum_{i=1}^{r-1} 2^{(r-i-1)/p} \Big\{
c_2 2^{(r-i)\vep /2}\cdot 2^{(r-i)(\vep-1)/2}\sigma(2^r)+ \prod_{l=i+1}^{r-1}
\big(1+p 2^{p-2} \rho^{2/p}(2^{l/p}_\ddag )\big)^{1/p} c_1 \sigma_1
\Big\} \\
& \le c_3 \bigg\{ \sum_{i=k}^{r-1} 2^{(r-i)(1/p-1/2+\vep)}\sigma(2^r) + \prod_{l=k+1}^{r-1}
\big(
1+p 2^{p-2} \rho^{2/p}(2^{l/p}_\ddag)
\big)^{1/p} c_1 \sigma_1 \sum_{i=1}^{r-1} 2^{(r-i-1)/p} \bigg\} \\
& \le c_4 \bigg\{
\sigma(2^r)+ \prod_{l=k+1}^{r-1} \big(
1+p 2^{p-2} \rho^{2/p}(2^{l/p}_\ddag)
\big)^{1/p} 2^{(r-k)/p}\sigma_1 
\bigg\} \\
& \le c_5 \sigma(2^r).
\end{align*}
Here, in the second step, since $\rho(\ell) \to 0$, we can choose $k$ large enough such that for all $i\ge k$,  
\begin{align}
\label{ineq:product:rho}
 \big(1+ p2^{p-2} \rho^{2/p}(2^{i/p}_\ddag )\big)^{1/p} \le 2^{\vep/2}, 
\end{align}
ensuring that $c_5=c_5(\rho,\vep)$. In the final step, we use the fact that 
\begin{align*}
 \prod_{l=k+1}^{r-1}\big(1+ p2^{p-2} \rho^{2/p}(2^{l/p}_\ddag ) \big)^{1/p}
2^{(r-k)/p} &\le 2^{(r-k)(\vep/2+1/p)} \\
& \le c_6 2^{(r-k)(1/2-\vep')}  
\end{align*} 
for some $\vep'\in (0,1/2)$, as $\vep/2 + 1/p < 1/2$. Constant $c_6$ also depends on $(\rho,p,\vep)$. 

Moreover, we observe that 
\begin{align*}
& \sum_{i=k}^{r-1} 2^{(r-i-1)/p} \prod_{l=i+1}^{r-1} 
\big(1+p 2^{p-2}\rho^{2/p}(2^{l/p}_\ddag )\big)^{1/p} 2 \cdot 2^{i/p}_\ddag a_1 \\
& \le  \prod_{l=k+1}^{r-1} \big(
1+p 2^{p-2} \rho^{2/p}(2^{l/p}_\ddag)
\big)^{1/p} 2^{r/p+1} \sum_{i=1}^{r-1} q_1 (2^i) a_1 \\
& \le  c_7 2^{r/p} \prod_{l=k+1}^{r-1} \big(
1+p 2^{p-2} \rho^{2/p}(2^{l/p}_\ddag)
\big)^{1/p}a_1,
\end{align*}
where we utilized $\sum_{i=1}^\infty q_1(2^i)<\infty$ and $c_7=c_7(q_1)$.  

The collection of these bounds yields that for all $r\ge k$ such that $k$ satisfies \eqref{ineq:product:rho},  
\begin{align}
\label{ineq:a2^4}
 a(2^r) \le c_8 \bigg\{
2^{r/p} \prod_{i=k}^{r-1}\big(
1+p 2^{p-2} \rho^{2/p}(2_\ddag^{i/p}) \big)^{1/p} a_1 + \sigma(2^r)
\bigg\} 
\end{align}
holds, where the constant $c_8$ depends on $(\rho,q_1,\vep)$, but $\vep$ depends only on $p$, i.e., $\vep \in (0,(p-2)/(2p))$. 
Hence, by replacing $c_8$ with 
\[
 c_9= \max \big\{ c_8,\,\max_{1\le l \le k} 2^{-l/p} a(2^l) / a_1 \big\},
\]
we obtain \eqref{ineq:a2^4} for all $r\ge 1$. 
Now, by writing $n$ in binary form and evaluating the error, we have 
\begin{align*}
 a_n \le c_{10} \bigg(n^{1/p} \prod_{i=1}^r \big(
1+p2^{p-2} \rho^{2/p}(2_\ddag^{i/p})
\big)^{1/p}a_1 +\sigma_n \bigg), 
\end{align*}
where we again notice that $\sigma_n=n^{1/2}h(n)$ for a slowly varying function $h$, so that 
$\sigma(2^r) \sim 2^{1/2} \sigma(2^{r-1})$ as $n\to\infty$.  
Finally, it suffices to observe that 
\begin{align*}
\prod_{i=1}^r \big(
1+ p2^{p-2} \rho^{2/p}(2^{i/p}_\ddag)
\big)^{1/p} & \le \exp\bigg( 1/p \sum_{i=1}^r \log \big(1+ p2^{p-2} \rho^{2/p}(2_\ddag^{i/p})\big) \bigg) \\
& \le \exp\big( 2^{p-2} \sum_{i=1}^r \rho^{2/p}(2^{i/p}_\ddag) \big).  
\end{align*}
Raising both sides of the bound for $a_n$ to the $p$-th power and recalling $a_n^p = \E|S_n|^p$ yields \eqref{inq:4th:moment}.
\end{proof}

\begin{remark}
\label{rem:for:lem:inq:4th:moment}
Since the function $q_1$ is flexible, we may choose $q_1$ such that $q_1(2^i) \ge 2^{-\vep/p \cdot i}$ holds for 
any small $\vep>0$ uniformly in $i\in \N$. With this in mind, we put $p=2+\delta$ in Eq. \eqref{inq:4th:moment}
for a sufficiently small $\delta>0$. Then the exponential term in Eq. \eqref{inq:4th:moment} is bounded 
as 
\begin{align*}
\exp \Big(
p 2^{p-2} \sum_{i=1}^r \rho^{2/p}(2_\ddag^{i/p}) 
\Big) &\le \exp \Big(
p 2^{p-2} \sum_{i=1}^r \rho^{2/p}(2^{(1-\vep)/p i})
\Big) \\
& \le c \exp \Big(
p^2 2^{p-2}/(1-\vep) \sum_{i=1}^r \rho^{2/p}(2^{ i})
\Big) \\
& \le c \exp \Big(
(2+\delta)^2 2^{\delta}/(1-\vep) \sum_{i=1}^r \rho^{2/(2+\delta)}(2^{ i})
\Big),
\end{align*}
where $r=[\log n]$, and we exploited a simple relation between sums and integrals 
together with a change of variables. By choosing $\delta$ and $\vep$ sufficiently small, the coefficient 
can be bounded by $5$. Thus, we obtain a version of \eqref{inq:4th:moment} as 
\begin{align*}
 \E|S_n|^{2+\delta} \le c \bigg(
n \E|X_0|^{2+\delta} \exp \Big(5 \sum_{i=1}^{[\log n]} \rho^{2/(2+\delta)}(2^{ i})\Big) + \sigma_n^{2+\delta}
\bigg). 
\end{align*}
Later we will use this inequality. 
The result would be a precise version of Lemma 2 of \cite{shao:1989} $($cf. Lemma 1 in \cite{shao:1988}$)$.
\end{remark}

\begin{proof}
As remarked in the proof of Theorem 1.1 of \cite{peligrad:1987}, for the strictly stationary series with 
\eqref{condition:Ibragimov}, the CLT is equivalent to the uniform integrability of $(S_n^2/\sigma_n^2)$. 
In \cite{peligrad:1982} it was proved the uniform integrability for the case $\sum_{i=1}^\infty \rho(2^i)<\infty$.
We shall prove the uniform integrability of 
$(S_n^2/\sigma_n^2)$ under the condition that $\sum_{i=1}^\infty \rho(2^i)=\infty$ with our condition. Recall that we assume 
\[
 g\Big(
n^{1/2} \big(h_\delta(n)h_d^{(2+\delta)/4}(n) \big)^{-1/\delta}\Big) \ge c h_d(n),\quad c>0 
\]
for every $n$ sufficiently large. For convenience we put $c=1$ without loss of generality. 
Set a threshold $T=T(n)$ such that
\[
 T = g^{\rm inv} \big(h_d(n)\big),
\] 
and separate each element of $(X_t)$ in the following way:
$X_t=X_{t,1}+X_{t,2}$ where 
\begin{align*}
 X_{t,1} &= X_t {\bf 1}(|X_t|\le T)- \E X_t {\bf 1}(|X_t|\le T), \\
 X_{t,2} &= X_t{\bf 1}(|X_t| > T) -\E X_t {\bf 1}(|X_t| > T) 
\end{align*}
and furthermore, we write 
\begin{align*}
& S_{n,1}= \sum_{t=1}^n X_{t,1}\quad \text{and}\quad S_{n,2}=\sum_{t=1}^n X_{t,2},  \\
& \sigma_{n,1}^2 = \var (S_{n,1})\quad \text{and}\quad \sigma_{n,2}^2 =\var (S_{n,2}). 
\end{align*}
By Lemma \ref{lemm:upper:bound} and the fact that $g(x)$ is non-decreasing function, 
we have
\[
 \sigma_{n,2}^2 \le \wt c_1 (n/g(T)) \E [X_0^2 g(|X_0|){\bf 1}(|X_0|>T)] \exp\Big(
\sum_{i=1}^r \rho (2_\dag^i)
\Big),
\]
where $\wt c_1$ and $2^{i}_\dag = [2^i \ov \ell (2^i)]$ are defined in Lemma \ref{lemm:upper:bound}, 
and $r=[\log n]$.

Recalling that the choice of $q$ and $q_0$ are flexible (Remark after Lemma \ref{lem:lower:bound})
 and the relation $\ov h (n)/(\wh c_1 \E X_0^2) \le \wh c_2 \ud h^{-1}(n)$, we may choose $q$ and $q_0$ 
such that $\ov \ell(n) \le \ud \ell(n)$ follows. 
Then 
\begin{align}
\label{ineq:up:low:ell}
 \exp\Big( \sum_{i=1}^r \rho (2_\dag^i)\Big) &\le \exp\Big(
\sum_{i=1}^{k_0-1} \rho (2_\dag^i)- \sum_{i=k_0}^r \log \big\{ 1-\rho\big(2_\ast^i\big) \big\}
\Big)\\
& \ll \exp\Big(
\sum_{i=1}^{k_0-1} \rho (2_\dag^i) \Big) g(T)^{1/2}, \nonumber
\end{align}
where $2_\dag^i \ge 2_\ast^i $.
By \eqref{ineq:up:low:ell} and the definition of $T$ it follows that for every $n$ sufficiently large 
\begin{align*}
 \sigma_{n,2}^2 
&\le (\wt c_1/\wt c_2) \exp\big( \sum_{i=1}^{k_0-1} \rho(2_\dag^i) \big)\, \sigma_n^2\, \E[X_0^2 g(|X_0|){\bf 1}(|X_0| >T)] \\
&\le c \sigma_n^2 \E[X_0^2 g(|X_0|){\bf 1}(|X_0| >T)]. 
\end{align*}
From this and the fact that $T\to\infty$ when $n\to\infty$, we obtain 
\begin{align}
\label{dominate:sigma2:sigma}
 \sigma_{n,2}=o(\sigma_n),\quad \text{as}\quad n\to\infty,
\end{align}
so that 
\[
 \sigma_{n,1} \sim \sigma_n\quad \text{as}\quad n\to\infty.
\]
Therefore we focus on $\sigma_{n,1}$.
By applying Lemma \ref{lem:inq:4th:moment} to the sequence $(X_{t,1})$ 
with $p=2+\delta$ (see also Remark \ref{rem:for:lem:inq:4th:moment}), we can find a constant 
$c=c(X,p,q_1)$ such that for every $n\ge 1$
\[
 \E |S_{n,1}|^{2+\delta} \le c \big(
n h_\delta (n)\E |X_0|^{2+\delta}{\bf 1}(|X_0| \le T) +\sigma_{n,1}^{2+\delta}
\big).
\]
Using Lemma \ref{lem:lower:bound}, we have 
\begin{align*}
 \E |S_{n,1}/\sigma_n|^{2+\delta} &\le c /\wt c_2^{1+\delta/2} \bigg\{
\frac{T^\delta \E|X_0|^2}{n^{\delta/2} h_\delta^{-1} h_d^{-(2+\delta)/4}}+1
\bigg\} \\
& =  c/\wt c_2^{1+\delta/2} 
\Bigg\{
\bigg(
\frac{T}{n^{1/2} \big( h_\delta 
h_d^{(2+\delta)/4} \big)^{-1/\delta}} \bigg)^{\delta}
\E |X_0|^2
+1
\Bigg\}
\end{align*}
Now recalling
\[
 g\Big(
n^{1/2} \big(h_\delta(n)h_d^{(2+\delta)/4}(n) \big)^{-1/\delta}\Big) \ge h_d(n) =g(T)
\]
we observe that 
\[
 T^\delta / \big\{ n^{\delta/2} \cdot h_\delta h_d^{(2+\delta)/4} \big\}
\]
is uniformly bounded on $n\ge n_0$ for some $n_0\in \N$.
Therefore, we conclude 
\begin{align}
\label{bound:sup:sn/sigma4}
 \sup_n \E |S_n/\sigma_n|^{2+\delta} <\infty.
\end{align}
Now from \eqref{dominate:sigma2:sigma} and \eqref{bound:sup:sn/sigma4} we can conclude that 
$(S_n^2/\sigma^2_n)$ is uniformly integrable. 
\end{proof}

\begin{proof}[Proof of Corollary \ref{cor:clt:log:mixing} $(b)$.]
In view of Theorem \ref{thm:main1}, the proof reduces to finding a non-decreasing function $g$ satisfying 
\eqref{condi:ii:improved} under condition $(b2)$ and the asymptotic relation 
\begin{align}
\label{asymp:g(n):col2}
    g(x) \sim \exp \big[ 2\alpha (2\log^+ x)^{1-\beta}/(1-\beta) \big].
\end{align}
To check \eqref{condi:ii:improved} 
we first evaluate $h_\delta(n)$ and $h_d(n)$ for large $n$. 
Applying the Taylor expansion 
\[
    -\log (1-x) = x + \frac{x^2}{2} + \frac{x^3}{3} + \frac{x^4}{4} + \cdots \le x + \frac{x^2}{2}\frac{1}{1-x}, \quad x \in (0,1),
\]
to $h_d(n)$ and noting that $\rho(2^i_\ast) < 1$ for $i \ge k_0$, we observe that 
\[
    h_d(n) \le \exp \bigg[ 2 \sum_{i=k_0}^{[\log n]} \Big(\rho(2^i_\ast) + \frac{\rho(2^i_\ast)^2}{2}\frac{1}{1-\rho(2^i_\ast)} \Big) \bigg], \quad n \ge 2^{k_0}.
\]
Here and what follows, we put $\rho(n)=\alpha (\log n)^{-\beta},\,n\ge 2^{k_0}$ without loss of generality 
since we need asymptotic rates of $h_\delta (n)$ and $h_d(n)$ as $n\to\infty$ and their coefficient constants do not matter. 
For sufficiently large $i$, i.e., $i \ge k_1$ for some $k_1 \in \N$, 
\begin{align}
    \rho(2^i_\ast) &= \rho([2^i \ud \ell (2^i)]) \nonumber \\
    &= \alpha \big(\log (2^i \ud \ell(2^i)-1)\big)^{-\beta} \nonumber \\
    &= \alpha \big\{ i + \log \ud \ell(2^i) + \log \big(1-2^{-i}\ud \ell^{-1}(2^i) \big) \big\}^{-\beta} \nonumber \\
    &\le \alpha i^{-\beta} (1-ci^{-\beta})^{-\beta}, \label{evluation:logell}
\end{align}
where the justification for the last inequality will be given at the end of the proof.

Using the inequality $(1-x)^{-\beta} \le 1+c x$ for some $c>0$ and $x \in (0,1)$, we further obtain 
\[
    \rho(2^i_\ast) \le \alpha i^{-\beta} + ci^{-2\beta}, \quad i \ge k_1,
\]
which, together with $\beta \in (1/2,1)$, yields 
\begin{align*}
    2 \sum_{i=k_0}^{[\log n]} \Big(\rho + \frac{\rho^2}{2}\frac{1}{1-\rho}\Big) 
    &\le 2 \alpha \sum_{i=k_1}^{[\log n]} i^{-\beta} + c \sum_{i=k_1}^{[\log n]} i^{-2\beta} + c \\
    &\le 2\alpha \sum_{i=k_1}^{[\log n]} i^{-\beta} + c.
\end{align*}
Thus, we obtain 
\begin{align}
\label{ineq:evalu:hdn}
    h_d(n) \le c \exp\Big( 2\alpha \sum_{i=k_0}^{[\log n]} i^{-\beta}\Big) \le c \exp \big( 2\alpha (\log n)^{1-\beta}/(1-\beta)\big). 
\end{align}
In a similar manner, we deduce 
\begin{align}
\label{ineq:evalu:hdelta}
    h_\delta^{1/\delta}(n) \le c \exp\Big( c (\log n)^{1-2\beta/(2+\delta)} \Big).
\end{align} 

Next, we specify the function $g$ as
\begin{align}
\label{def:g(n):col2}
    g(x) = \exp \Big( 2\alpha \sum_{i=1}^{[\log x^2]} i^{-\beta} \Big) 
    \sim \exp\Big( 2\alpha (2\log n)^{1-\beta}/(1-\beta) + c \Big), 
\end{align}
where the middle expression is convenient for our purpose. Then 
\begin{align*}
    g\Big( n^{1/2}\big(h_\delta(n)h_d^{(2+\delta)/4}(n)\big)^{-1/\delta}\Big) 
    &= \exp \Big( 2\alpha \sum_{i=1}^{[\log n-e_n]} i^{-\beta} \Big) \\
    &= \exp \Big( 2\alpha \sum_{i=1}^{[\log n]} i^{-\beta} - 2\alpha \sum_{i=[\log n-e_n]+1}^{[\log n]} i^{-\beta} \Big) \\
    &=: \exp \Big( 2\alpha \sum_{i=1}^{[\log n]} i^{-\beta} - E_n \Big)\\
    &\sim c \exp \Big( 2\alpha (\log n)^{1-\beta}/(1-\beta) - E_n \Big), 
\end{align*}
where 
\[
 e_n = 2/\delta \log h_\delta(n) +(2+\delta)/(2\delta) \log h_d(n).
\]
We will show that $E_n$ is negligible. This fact, together with \eqref{ineq:evalu:hdn}, ensures that $g$ defined in \eqref{def:g(n):col2} satisfies both \eqref{condi:ii:improved} and \eqref{asymp:g(n):col2}, thereby completing the proof. 

From \eqref{ineq:evalu:hdn} and \eqref{ineq:evalu:hdelta}, we have 
\begin{align*}
    e_n \le c (\log n)^{1-2\beta/(2+\delta)} + c,
\end{align*}
and thus, 
\begin{align*}
    E_n &\le 2 \alpha \cdot e_n \cdot c (\log n)^{-\beta} \\
    &\le c (\log n)^{1-\beta(4+\delta)/(2+\delta)} + c (\log n)^{-\beta}. 
\end{align*}
Since $\delta > 0$ can be chosen arbitrarily small, we may let 
\[
    \frac{4\beta - 2}{1-\beta} > \delta \quad \Leftrightarrow \quad 1 - \beta \frac{4+\delta}{2+\delta} < 0.
\]
For such $\delta > 0$, it follows that 
\[
    E_n \to 0 \quad \text{as} \quad n \to \infty.
\]

Finally, it remains to verify the bound used in \eqref{evluation:logell}: 
\[
    \big|\log \ud \ell(2^i) + \log \big(1-2^{-i}\ud \ell^{-1}(2^i)\big)\big| \le c i^{1-\beta},
\]
for every $i$ sufficiently large. Note that $\ud \ell$ is a non-increasing slowly varying function, so that 
$2^{-i} \ud \ell^{-1}(2^i) \le 2^{-i(1-\vep)} < 1$ for any $\vep > 0$ and sufficiently large $i$. Consequently, the second term is bounded by a positive constant. Moreover, recall that $\ud \ell(n) = \ud h^4(n)q_0(n)$. In view of the setting in Lemma \ref{lem:inq:4th:moment}, we may take $q_0(n) = c (\log n)^{-4}$. Thus, by the definition of $\ud h$, we have 
\begin{align*}
    \big|\log \ud \ell (2^i)\big| &\le 4 \big|\log \ud h(2^i) \big| + \big| \log q_0(2^i) \big| \\
    &\le 4\alpha \sum_{j=1}^{[(1-\vep_1)i]} (\log 2^j)^{-\beta}/(1-\vep_2) + c\log i + c \\
    &\le c \sum_{j=1}^i j^{-\beta} + c\log i \\
    &\le c i^{1-\beta}. 
\end{align*}
\end{proof}

\appendix
\section{Auxiliary results}
\label{appendix}
\begin{lemma}
\label{aux:lem:ineq}
For all $x,y\ge 0$ and $a\ge 1$, the inequality
\begin{align*}
(x+y)^a \le x^a +y^a +a (2^{a-2}\vee 1) (x^{a-1}y +xy^{a-1})
\end{align*}
holds. 
\end{lemma}

\begin{proof}
The inequality trivially holds for $x=0$, so without loss of generality we may assume $x>0$. 
Dividing both sides by $x^a$ and replacing $y/x$ with $x$, it suffices to show that
\[
(1+x)^a \le 1+x^a + a(2^{a-2}\vee 1) (x+x^{a-1}),\quad x\ge 0.
\]
We define 
\[
f(x) = 1+ x^a +a(2^{a-2}\vee 1)(x+x^{a-1}) -(1+x)^a
\]
and show that $f(x)\ge 0$ for all $x\ge 0$. Observing the symmetry property 
$f(1/x)=x^{-a}f(x)$, it is enough to consider $f$ on either $x\in [0,1]$ or $x\in [1,\infty)$. 

First, consider the case $a\in (1,3]$ and $x\in [0,1]$. Note that $f(0)=0$ and $f(1)=2+2a(2^{a-2}\vee 1)-2^a>0$. The case $a=1$ is trivial. 
We analyze the sign of the derivative
\[
f'(x) = a x^{a-1} + a(2^{a-2}\vee 1)\{1+(a-1) x^{a-2}\}-a(1+x)^{a-1}. 
\]
When $a \in (1,2]$, the triangle inequality $(1+x)^{a-1} \le 1+ x^{a-1}$ immediately yields $f'(x) \ge 0$. 
For $a \in (2,3]$, by the convexity of $(1+x)^{a-1}$, we have 
\[
(1+x)^{a-1} \le 1 +(a-1)(1+x)^{a-2} x \le 1+(a-1) 2^{a-2} x,
\]
which implies that
\[
f'(x) \ge a x^{a-1}+ a(2^{a-2}-1)+a(a-1) 2^{a-2} (x^{a-2}-x) >0.
\]
Hence, $f(x)\ge 0$ is proved for $a\in (1,3]$.

Next, suppose $a \ge 3$. We examine $f$ on $x\in [1,\infty)$, and observe that 
$f(1)= 2+a2^{a-1}-2^a>0$ and $f(\infty)=\infty$ (by applying L'H\^opital's rule to $f(1/x)=x^{-a}f(x)$). 
Moreover, $f'(x)\ge 0$ holds on $[1,\infty)$. Indeed, applying the mean value theorem to $(1+z)^{a-1}$ for $z\in [0,1]$, we obtain 
\begin{align*}
f'(x) &= a x^{a-1} \Big\{
1+2^{a-2}\big(x^{1-a}+(a-1)x^{-1}\big)-(1+x^{-1})^{a-1}
\Big\} \\
&\ge a x^{a-1} \Big\{
(a-1) \big(2^{a-2} -(1+x_\ast^{-1})^{a-2}\big)x^{-1} + 2^{a-2}x^{1-a}
\Big\} \\
& \ge a 2^{a-2}, 
\end{align*} 
where $x_\ast$ is some value in $[1,\infty)$. This completes the proof.
\end{proof}

\end{document}